\documentclass[11pt,reqno]{amsart}
\usepackage{amssymb,mathrsfs,graphicx,subfigure, enumerate}
\usepackage{amsmath,amsfonts,amssymb,amscd,amsthm,bbm}
\usepackage{extpfeil}
\usepackage{graphicx,colortbl}
\usepackage{float}
\usepackage{epsfig}
\usepackage{caption}
\usepackage{bm}
\graphicspath{{Figure/}}

\usepackage[margin=1in]{geometry}

\makeatletter
\@namedef{subjclassname@2020}{\textup{2020} Mathematics Subject Classification}
\makeatother

\title[Navier-Stokes equations]{The method of $a$-contraction with shifts used for long-time behavior toward viscous shock}

\author[Han]{Sungho Han}
\address[Sungho Han]{\newline Department of Mathematical Sciences \newline Korea Advanced Institute of Science and Technology, Daejeon  34141, Republic of Korea}
\email{sungho\_han@kaist.ac.kr}

\author[Kang]{Moon-Jin Kang}
\address[Moon-Jin Kang]{\newline Department of Mathematical Sciences \newline Korea Advanced Institute of Science and Technology, Daejeon  34141, Republic of Korea}
\email{moonjinkang@kaist.ac.kr}

\author[Lee]{Hobin Lee}
\address[Hobin Lee]{\newline Department of Mathematical Sciences \newline Korea Advanced Institute of Science and Technology, Daejeon  34141, Republic of Korea}
\email{lcuh11@kaist.ac.kr}

\begin{document}
	\newtheorem{theorem}{Theorem}[section]
	\newtheorem{lemma}{Lemma}[section]
	\newtheorem{corollary}{Corollary}[section]
	\newtheorem{proposition}{Proposition}[section]
	\newtheorem{remark}{Remark}[section]
	\newtheorem{definition}{Definition}[section]
	
	\renewcommand{\theequation}{\thesection.\arabic{equation}}
	\renewcommand{\thetheorem}{\thesection.\arabic{theorem}}
	\renewcommand{\thelemma}{\thesection.\arabic{lemma}}
	\newcommand{\bbr}{\mathbb R}
	\newcommand{\bbz}{\mathbb Z}
	\newcommand{\bbn}{\mathbb N}
	\newcommand{\bbs}{\mathbb S}
	\newcommand{\bbp}{\mathbb P}
	\newcommand{\bbt}{\mathbb T}
	\newcommand{\<}{\langle}
	\renewcommand{\>}{\rangle}
	\newcommand{\T}{\mathbb{T}}
	\newcommand{\N}{\mathbb{N}}
	\newcommand{\R}{\mathbb{R}}
	\newcommand{\lt}{\left}
	\newcommand{\rt}{\right}
	\newcommand{\bq}{\begin{equation}}
	\newcommand{\eq}{\end{equation}}
	\newcommand{\e}{\varepsilon}
	\newcommand{\mc}{\mathcal{C}}
	\newcommand{\pa}{\partial}
	\newcommand{\tU}{\widetilde{U}}
	\newcommand{\tu}{\widetilde{u}}
	\newcommand{\tv}{\widetilde{v}}
	\newcommand{\tw}{\widetilde{w}}
	\newcommand{\pv}{p(v)}
	\newcommand{\tp}{\widetilde{p}}
	\newcommand{\tpv}{p(\widetilde{v})}
	\newcommand{\norm}[1]{\left\lVert#1\right\rVert}
	\newcommand{\beq}{\begin{equation}}
	\newcommand{\eeq}{\end{equation}}

	
	
	
	\thanks{\textbf{Acknowledgment.} S. Han, M.-J. Kang and H. Lee were partially supported by the National Research Foundation of Korea  (NRF-2019R1C1C1009355) }

	\begin{abstract} 
		We revisit the method of $a$-contraction with shifts used for long-time behavior of barotropic Navier-Stokes flows perturbed from a Riemann shock. 
		For the usage of the method of $a$-contraction with shifts, we do not employ the effective velocity $h$ variable even for higher order estimates.
		This approach would be important when handling the barotropic Navier-Stokes system with other effects, for example, such as capillary effect and boundary effect. 
	\end{abstract}
	
	\maketitle
	
	\tableofcontents
	
	\section{Introduction}\label{sec:1}
	\setcounter{equation}{0}
	
	The method of $a$-contraction (with shifts) has been extensively used for orbital stability of shock waves measured by the weighted (by $a$) relative entropy up to dynamical shifts.  
	This method was developed by the second author and Vasseur \cite{KV16,Vasseur-2013} for the contraction property of perturbations of extremal shocks in the hyperbolic system of conservation laws, especially for the Euler system. That was extended to the viscous hyperbolic conservation laws: first for the scalar equation in 1D case \cite{Kang-V-1}  (for more general case \cite{Kang19}), and then in multi-D case \cite{KVW,KO}. As the first extension of the method to the Navier-Stokes system, Kang-Vasseur \cite{KV21} prove the contraction of any large disturbances for a single viscous shock. It is shown in \cite{KV-Inven} that this contraction estimate is uniform in the strength of viscosity, and ensures the existence of inviscid limits on which the associated Riemann shock is orbitally stable (see also \cite{KV-2shock}). 
	 Surprisingly, the method of $a$-contraction with shifts has been used in studies of models in other contexts \cite{CKKV,CKV,HKKL}.
	 
	The $a$-contraction method is also used for time-asymptotic stability of Navier-Stokes (NS) flows slightly perturbed in $H^s (s\ge1)$ from Riemann data (generating Riemann shock).
	 The result \cite{KVW23} shows the long-time behavior of the barotropic NS system towards composition of shock and rarefaction, and the result by Wang-Wang \cite{WW} handles the behavior of multi-D perturbation towards a single shock. Its extension to the case of generic Riemann solution for the Navier-Stokes-Fourier system was studied in \cite{KVW-NSF}. \\

In this paper, we revisit the $a$-contraction method used in the proof for the long-time behavior towards a viscous shock for the one-dimensional barotropic Navier-Stokes (NS) equations, written in the Lagrangian mass coordinates (for simplicity):
	\begin{equation}\label{eq:NS}
	\begin{aligned}
	&v_t-u_x=0,\quad (t,x)\in \bbr_+\times\bbr,\\
	&u_t+p(v)_x=\left(\mu(v) \frac{u_x}{v}\right)_x,
	\end{aligned}
	\end{equation}
	where the unknown functions $v=v(t,x)$ and $u=u(t,x)$ represent the specific volume and velocity of  fluid, respectively. 
	The pressure $p=p(v)$ is assumed to satisfy the $\gamma$-law, that is,
	\begin{equation*}
	p(v)=bv^{-\gamma}, \quad b>0, \quad \gamma>1.
	\end{equation*}
	 The initial data of the NS system \eqref{eq:NS} is given by $(v_0,u_0)$, whose far-field states are prescribed as constants:
	\begin{equation}\label{farcon}
	\lim_{x \to \pm \infty} (v_0(x),u_0(x))=(v_\pm,u_\pm).
		\end{equation}
	 Based on hueristic argument in \cite{M18}, the time-asymptotic profile of the NS equations is related to the Riemann solution to the Euler equations:
	\begin{equation}\label{eq:Euler}
	\begin{aligned}
	&v_t-u_x=0,\quad (t,x)\in\bbr_+\times\bbr,\\
	&u_t+p(v)_x=0,
	\end{aligned}
	\end{equation}
	subject to the Riemann initial data
	\begin{equation}
	\label{Riemann-data}
	(v(0,x),u(0,x))=
	\begin{cases} 
	(v_-,u_-), &x<0,\\
	(v_+,u_+), &x>0.
	\end{cases}
	\end{equation} 
     We focus on the case when the end states $(v_\pm,u_\pm)$  are connected by a single Hugoniot curve. Without loss of generality, we only handle the case of a 2-shock curve. In other words, for a given right-end state $(v_+,u_+)$, we consider the left-end state $(v_-,u_-)$ satisfying the following Rankine-Hugoniot conditions:
	\begin{align}\label{RH}
	\begin{cases}
	-\sigma(v_+-v_-)-(u_+-u_-)=0,\\
	-\sigma(u_+-v_-)-(p(v_+)-p(v_-))=0,
	\end{cases}
	\quad \sigma:=\sqrt{-\frac{p(v_+)-p(v_-)}{v_+-v_-}}>0,
	\end{align}
	and the entropy condition $v_-<v_+, \quad u_->u_+$.\\
	Then the Riemann solution $(\overline{v},\overline{u})$ to the Euler equations \eqref{eq:Euler}--\eqref{Riemann-data} is given by 2-shock wave
	\begin{equation}\label{Riemann-solution}
	(\overline{v}(t,x),\overline{u}(t,x))=
	\begin{cases}
	(v_-,u_-) &\text{if} \quad x<\sigma t,\\
	(v_+,u_+) &\text{if} \quad x>\sigma t.
	\end{cases}
	\end{equation}
	For the case of NS equations \eqref{eq:NS}, the counterpart of the Riemann solution \eqref{Riemann-solution} is a viscous shock, as a traveling wave $(\tv,\tu)(x-\sigma t)$ solution to \eqref{eq:NS}, that satisfies the following ODEs:
	\begin{align}
	\begin{aligned}\label{viscous-shock}
	&-\sigma \tv'-\tu'=0,\\
	&-\sigma \tu'+p(\tv)'=\left( \frac{\tu'}{\tv}\right)',\\
	&(\tv,\tu)(\pm\infty)=(v_\pm,u_\pm).
	\end{aligned}
	\end{align}

In previous works \cite{HKK23,KV21,KV-Inven,KV-2shock,KVW23,WW} on the barotropic NS system, the $a$-contraction method has been used in showing that perturbations of a viscous shock, measured by the weighted relative entropy as the zeroth order estimate, is not increasing in time up to a dynamical shift. For that, the effective velocity $h:=u-\frac{\mu(v)}{v}v_x$ (corresponding to the Bresch-Desjardins entropy) was used to transform the original system \eqref{eq:NS} into a new system that has a viscous term on $v$ variable only.  Using the new system with $h$ variable seems to make the $a$-contraction method to be applied easily, because of the following two reasons: (i) the hyperbolic part of the system is linear in $u$ (or $h$) but nonlinear in $v$ due to the pressure; (ii) the main terms $\int a' (p-p(\tv)) (h-\tilde h) dx - \frac{\sigma}{2} \int a' (h-\tilde h)^2 dx$ can be maximized explicitly (without lower order terms), contrary to the maximization process for the original case \eqref{eq:NS} as in Section 4.2.

However, when handling the long-time behavior for a slightly complicated problem than \eqref{farcon}, using the $h$ variable for the zeroth order estimate may not be helpful and even could make the usage of the method to be complicated. 
For example, for the Navier-Stokes system of Korteweg type as in \cite{HKKL}, the associated new system (represented by $h$ variable) would be far more complicated than the original system since the Korteweg term (or capillarity term) is substantially related to the third derivatives of $v$ variable.
For the boundary-initial value problem of barotropic Navier-Stokes as in \cite{XHKKL}, handling the original system is easier than employing the $h$ variable because of no boundary condition for the derivative of $v$ variable.  \\

The goal of the paper is to present a way of applying the $a$-contraction method to the NS system \eqref{eq:NS} without using the $h$ variable for all estimates, in order for the  long-time behavior toward a viscous shock.

For simplicity, we assume \eqref{eq:NS} with a constant viscosity coefficient $\mu>0$ as a normalized coefficient $\mu=1$ and also $b=1$.
	However, out result still holds for \eqref{eq:NS} with a general pressure $p(v)>0$ satisfying $p'(v) < 0, p''(v) > 0$, and smooth viscosity $\mu=\mu(v)$  for $v>0$,
	since we consider small $H^1$-perturbations for $v$ variables.

	Our main theorem reads as follows.
	
	\begin{theorem}\label{thm:main}
    For a given state $(v_+,u_+)\in\bbr^+\times\bbr$, there exist positive constants $\delta_0$, and $\e_0$ such that the following holds:

    For any $(v_-,u_-)$ on the 2-shock curve $S_2(v_+,u_+)$, satisfying the Rankine-Hugoniot condition \eqref{RH} with $|v_+-v_-|<\delta_0$, consider the 2-viscous shock $(\tv,\tu)(x-\sigma t)$ defined in \eqref{viscous-shock}. Let $(v_0,u_0)$ be any initial data such that
    \begin{equation*}
    \sum_{\pm}\left(\|(v_0-v_{\pm},u_0-u_{\pm})\|_{L^2(\R_\pm)}\right)+\|(v_{0x},u_{0x})\|_{L^2(\R)}<\e_0,
    \end{equation*}
    where $\R_-:=-\R_+=(-\infty,0)$. Then, the Navier-Stokes system \eqref{eq:NS} admits a unique global-in-time solution $(v,u)$. Moreover, there exists a Lipschitz continuous shift $X(t)$ such that
    \begin{equation*}
    \begin{aligned}
    &(v(t,x)-\tv(x-\sigma t-X(t))\in C(0,\infty;H^1(\R)),\\
    &u(t,x)-\tu(x-\sigma t-X(t))\in C(0,\infty;H^1(\R)),\\
    &u_{xx}(t.x)-\tu_{xx}(x-\sigma t -X(t)) \in L^2(0,\infty; L^2(\mathbb{R}))
    \end{aligned}
    \end{equation*}
    In addition, we have
    \begin{equation}
    \lim_{t\to\infty}\sup_{x\in\R}\left|(v,u)(t,x)-(\tv,\tu)(x-\sigma t-X(t))\right|=0
    \end{equation}
    and
    \begin{equation}\label{xlimit}
    \lim_{t\to\infty} |\dot{X}(t)|=0.
    \end{equation}
	\end{theorem}

	\begin{remark}
		Since \eqref{xlimit} implies
		$$
		\lim_{t\rightarrow+\infty}\frac{X(t)}{t}=0,
		$$
		the shift function $X(t)$ grows at most sub-linearly as $t\to\infty$. Thus, the shifted wave $\tU(x-\sigma t-X(t))$ tends to the original wave $\tU(x-\sigma t)$ time-asymptotically.
	\end{remark}

	\begin{remark}
		 There are a lot of previous literatures on the time asymptotic behavior of the NS equations. Among the numerous results on the time-asymptotic stability of the NS equations \eqref{eq:NS}, we refer to \cite{G86,HKK23,KVW23,MN85,MN86,MW10}, although the list is not exhaustive. In particular, Matsumura-Nishihara \cite{MN86} used the  anti-derivative method for the time-asymptotic stability of viscous shock wave, for which the zero-mass condition for the initial perturbation is crucially imposed. Our result does not require the zero-mass condition. 
	\end{remark}
	
	The rest of the paper is organized as follows. In Section \ref{sec:prelim}, we provide several preliminaries, such as technical estimates on the relative quantities or the properties of the viscous shock \eqref{viscous-shock}.  Section \ref{sec:apriori} provides the a priori estimate on the perturbation, which guarantees the global existence of the solution to the NS equation, as well as the time-asymptotic behavior of the solution. Then, we focus on proving a priori estimate. In Section \ref{sec:rel_ent}, we obtain $L^2$ estimates by the method of $a$-contraction with shift, and then we obtain the estimates on the high-order terms in Section \ref{sec:high-order}.

	\section{Preliminaries}\label{sec:prelim}
	\setcounter{equation}{0}
	In this section, we present several preliminary estimates on the relative quantities for the pressure and the internal energy. We also provide the existence and properties of viscous shock in this section. Finally, we introduce several $O(1)$-constants and related estimates on them.
	
	\subsection{Estimates on the relative quantities}
	We present several upper and lower bounds on the relative quantities that will be used in estimating the relative entropy. For any function $F:(0,\infty)\to \bbr$ and $v,w\in (0,\infty)$, we define the relative quantity $F(v|w)$ as
	\[F(v|w):=F(v)-F(w)-F'(w)(v-w).\]
	In particular, when $F$ is convex, then the relative quantity is always positive. In the following lemma, we present several lower and upper bounds on the relative quantities for the pressure $p(v)=v^{-\gamma}$ and the internal energy $Q(v)=\frac{v^{1-\gamma}}{\gamma-1}$.
	
	\begin{lemma}\label{lem : Estimate-relative} Let $\gamma>1$ and $v_+$ be given constants. Then, there exists constants $C, \delta_*$ such that the following assertions hold:
		\begin{enumerate}
			\item For any $v,\bar{v}$ satisfying $0<\bar{v}<2v_+$ and $0<v<3v_+$,
			\begin{equation*}
			|v-\bar{v}|^2 \le C Q(v|\bar{v}), \quad |v-\bar{v}|^2 \le C p(v|w).
			\end{equation*}
			\item For any $v,\bar{v}$ satisfying $v,\bar{v} > v_+ /2$,
			\begin{equation*}
			|p(v)-p(\bar{v})| \le C |v-\bar{v}|.
			\end{equation*}
			\item For any $0<\delta<\delta_*$ and any $(v,\bar{v}) \in \mathbb{R}^2_+$ satisfying $|p(v)-p(\bar{v})|<\delta$ and $|p(\bar{v})-p(v_+)| < \delta,$
			\begin{equation*}
			\begin{aligned}
			&p(v|\bar{v}) \le \left( \frac{\gamma +1 }{2 \gamma } \frac{1}{p(\bar{v})} +C \delta \right) |p(v)-p(\bar{v})|^2,\\ 
			&Q(v|\bar{v}) \ge \frac{p(\bar{v})^{- \frac{1}{\gamma}-1}}{2 \gamma}|p(v)-p(\bar{v})|^2- \frac{1+\gamma}{3\gamma^2}p(\bar{v})^{- \frac{1}{\gamma}-2}(p(v)-p(\bar{v}))^3,\\ 
			&Q(v|\bar{v}) \le \left( \frac{p(\bar{v})^{- \frac{1}{\gamma}-1}}{2 \gamma} +C \delta \right)|p(v)-p(\bar{v})|^2.
			\end{aligned}
			\end{equation*}
			
		\end{enumerate}	
	\end{lemma}
	
	\begin{proof}
		Since the proofs are duplicates of those of \cite[Lemma 2.4, 2.5, and 2.6]{KV21}, we omit the proof.
	\end{proof}

	\subsection{Viscous shock wave} 
	In the following lemma, we list the properties of the 2-viscous shock wave $(\tv,\tu)(\xi)$. The proof of this lemma can be found in \cite{MN85} and \cite{G86}.

	\begin{lemma} \label{lem:shock-property}
		For a given right-end state $(v_+,u_+)$, there exists a positive constant $C>0$ such that the following statement holds. For any left end state $(v_-,u_-) \in S_2(v_+,u_+)$ with $|v_+-v_-| \sim |u_+ - u_-|=:\delta<\delta_0$, there exists a unique solution $(\tv,\tu)(\xi)$ to \eqref{viscous-shock} such that $\tv(0)=\frac{v_-+v_+}{2}$. Moreover, the following estimates hold: 
		\begin{equation}
		\begin{aligned}\label{shock-property}
		& \tu'<0,\quad \tv'>0, \quad '=\dfrac{d}{d\xi}, \;\xi=x-\sigma t,\\
		&C^{-1} \tv'(\xi) \le \tu'(\xi) \le C \tv' (\xi), \quad \xi \in \mathbb{R},\\
		&|\tv(\xi)-v_\pm|\le C\delta e^{-C\delta|\xi|},\quad \pm \xi>0,\\
		&|\tv'(\xi)|\le C\delta^2e^{-C\delta|\xi|},\quad |\tv''(\xi)|\le C\delta|\tv'(\xi)|.
		\end{aligned}
		\end{equation}
	\end{lemma}

	\subsection{Useful $O(1)$-constants}
	In the later analysis, we will use the following $O(1)$-constants defined as
	\begin{equation} \label{O(1) constnats}
	\sigma_\ell:=\sqrt{-p'(v_-)}, \quad \alpha_\ell := \frac{\gamma+1}{2 \gamma \sigma_\ell p(v_-)}=\frac{p''(v_-)}{2|p'(v_-)|^2 \sigma_\ell}.
	\end{equation}
	These constants are indeed independent of the small shock strength $\delta$ since $v_+/2 \le v_- \le v_+$. Then, the following estimates on the $O(1)$-constants hold:
	\begin{equation}\label{shock_speed_est}
	|\sigma -\sigma_\ell|=\left|\sqrt{-\frac{p(v_+)-p(v_-)}{v_+-v_-}}-\sqrt{-p'(v_-)}\right| \le C \delta.
	\end{equation}
	Moreover, thanks to Lemma \ref{lem:shock-property}, the shock profile is monotone for the weak shock, and therefore $v_-\le \tv(x-\sigma t)\le v_+$ for all $x\in \R$. This yields the following estimates
	\begin{align}
	\begin{aligned}\label{shock_speed_est-2}
	&\|\sigma_\ell^2+p'(\tv)\|_{L^\infty}=\|p'(\tv)-p'(v_-)\|_{L^\infty} \le C \delta, \\
	&\left\| \frac{1}{\sigma_\ell^2}-\frac{p(\tv)^{-\frac{1}{\gamma}-1}}{\gamma} \right\|_{L^\infty}=\left\|\frac{(v_-)^{\gamma+1}}{\gamma }-\frac{p(\tv)^{-\frac{1}{\gamma}-1}}{\gamma}\right\|_{L^\infty} \le C \delta.
	\end{aligned}
	\end{align} 
	Throughout the paper, $C$ denotes a positive $O(1)$-constant which may change from line to line, but which is independent of the small constants like $\delta, \e_1$ and the lifespan $T$ given in Proposition \ref{apriori-estimate}.\\

	\section{A priori estimate and Proof of Theorem \ref{thm:main}}\label{sec:apriori}
	\setcounter{equation}{0}
	
	In this section, we first provide the a priori estimate for the perturbation, which is the key estimate for the main theorem. The proof of a priori estimate is presented in the next two sections. After stating the a priori estimate, we prove the global existence and time-asymptotic behavior of the solution, completing the proof of Theorem \ref{thm:main}.
	
	\subsection{Local existence}
	
	We first provide the local existence of strong solutions to the original NS system \eqref{eq:NS}.
	
	\begin{proposition}\label{prop:local}
		Let $\underbar{v}$ and $\underbar{u}$ be smooth monotone functions such that
		\[\underbar{v}(x) = v_{\pm},\quad\underbar{u}(x)=u_{\pm},\quad\mbox{for}\quad\pm x\ge 1.\]
		Then, for any constants $M_0$, $M_1$, $\underline{\kappa}_0$, $\overline{\kappa}_0$, $\underline{\kappa}_{1}$, and $\overline{\kappa}_1$ with
		\[0<M_0<M_1,\quad\mbox{and}\quad0<\underline{\kappa}_1<\underline{\kappa}_0<\overline{\kappa}_0<\overline{\kappa}_1,\]
		there exists a finite time $T_0>0$ such that if the initial data $(v_0,u_0)$ satisfy
		\[\|v_0-\underline{v}\|_{H^1(\R)}+\|u_0-\underline{u}\|_{H^1(\R)}\le M_0,\quad\mbox{and}\quad \underline{\kappa}_0\le v_0(x)\le\overline{\kappa}_0,\quad\forall x\in\R,\]
		the Navier-Stokes equations \eqref{eq:NS} admit a unique solution $(v,u)$ on $[0,T_0]$ satisfying
		\begin{align*}
		v-\underline{v}\in L^\infty ([0,T_0];H^1(\R)),\quad u-\underline{u}\in L^\infty([0,T_0];H^1(\R))\cap L^2([0,T_0];H^2(\R)),
		\end{align*}
		\[\|v-\underline{v}\|_{L^\infty([0,T_0];H^1(\R))}+\|u-\underline{u}\|_{L^\infty([0,T_0];H^1(\R))}\le M_1\]
		and
		\[\underline{\kappa}_1\le v(t,x)\le \overline{\kappa}_1,\quad \forall(t,x)\in [0,T_0]\times \R.\]
	\end{proposition}
	
	\begin{proof}
		The proof of the local existence can be obtained by using the standard argument of generating a sequence of approximate solutions and the Cauchy estimate, see for example \cite{S76} and \cite{MV}. For the brevity of the paper, we omit the proof. 
	\end{proof}
		\subsection{Construction of weight function}
 We first construct the weight function $a=a(t,x)$ as
	\begin{equation} \label{a}
	a(t,x):=1+\frac{u_--\widetilde{u}(x-\sigma t)}{\sqrt{\delta }},
	\end{equation}
	where $\delta :=|u_- - u_+| $ denotes the shock strength. It follows from the definition of the weight function $a$ that $1 \le a \le 1+\sqrt{\delta }<\frac{3}{2}$ and
	\begin{equation}\label{a_x}
	\partial_x a=-\frac{\partial_x \tu}{\sqrt{\delta }}=\frac{\sigma\pa_x \tv}{\sqrt{\delta }}>0,\quad \mbox{and therefore},\quad |a_x|\sim \frac{|\pa_x\tv|}{\sqrt{\delta }}
	\end{equation}
	where we used \eqref{viscous-shock}$_1$.
	\subsection{Construction of shift}
	Next, we introduce the shift $X:\bbr_+\to\bbr$ as a solution to the following ODE:
	\begin{align}
	\begin{aligned}\label{ODE_X}
	\dot{X}(t)&=-\frac{M}{\delta}\Bigg(\int_{\R} a\left(x- \sigma t- X(t)\right)\widetilde{u}_x\left(x- \sigma t- X(t)\right)\big(u-\widetilde{u}(x- \sigma t- X(t)\big)\,d x\\
	&\hspace{2cm}+\frac{1}{\sigma}\int_\R a(x- \sigma t- X(t))\pa_xp\big(\widetilde{v}(x- \sigma t- X(t))\big)\big(u-\widetilde{u}(x- \sigma t- X(t))\big)\,d x\Bigg),\\
 X(0)&=0,
	\end{aligned}
	\end{align}
	where $M=\frac{5\sigma_\ell^3\alpha_\ell}{4}$. Then, the standard existence theorem for the ODE can be applied to guarantee the existence of the shift.
	
	\begin{proposition}
		For any $c_1,c_2,c_3>0$, there exists a constant $C>0$ such that the following is true. For any $T>0$, and any function $v,u\in L^\infty((0,T)\times\R)$ with
		\[c_1\le v(t,x)\le c_2,\quad |u(t,x)|\le c_3,\quad(t,x)\in[0,T]\times\R,\]
		the ODE \eqref{ODE_X} has a unique Lipschitz continuous solution $X$ on $[0,T]$. Moreover, we have
		\[|X(t)|\le Ct,\quad t\in[0,T].\]
	\end{proposition}
	
	As the name implies, the constructed shift $X(t)$ will play an important role in the theory of $a$-contraction with shift. In the following, we use the following abbreviated notation for the shifted function. For any function $g:\bbr\to\bbr$, we define
	\[g^X(\cdot):=g(\cdot-X(t)),\quad t\ge0.\]
	
	\subsection{A priori estimate}
	We now state the a priori estimate, which is the key estimate for obtaining the time-asymptotic behavior of the NS equations.
	
	\begin{proposition}\label{apriori-estimate}
		For a given state $(v_+,u_+)\in\bbr^+\times\bbr$, there exist positive constants $C_0,\delta_0$, and $\e_1$ such that the following holds:
		
		Suppose that $(v,u)$ is the solution to \eqref{eq:NS} on $[0,T]$ for some $T>0$, and $(\tv,\tu)$ is defined in \eqref{viscous-shock}. Let $X$ be the Lipschitz continuous solution to \eqref{ODE_X} with weight function $a$ defined in \eqref{a}. Assume that the shock strength $\delta$ is less than $\delta_0$ and that
		\begin{align*}
		&v-\tv^X\in L^\infty(0,T;H^1(\bbr)),\\
		&u-\tu^X\in L^\infty(0,T;H^1(\bbr))\cap L^2(0,T;H^2(\bbr)),
		\end{align*}
		and
		\begin{equation}\label{smallness}
		\|v-\tv^X\|_{L^\infty(0,T;H^1(\bbr))}+\|u-\tu^X\|_{L^\infty(0,T;H^1(\bbr))}\le \e_1.
		\end{equation}
		Then, for all $0\le t\le T$,
		
		\begin{equation}
		\begin{aligned}\label{a-priori-1}
		&\sup_{t\in[0,T]}\left(\norm{v-\tv^X}_{H^1(\mathbb{R})}^2 +\norm{u-\tu^X}_{H^1(\mathbb{R})}^2 \right)+\delta \int_0^t | \dot{X}(s)|^2 \, d s \\ 
		&\quad +\int_0^t \left( G_1+G^S \right) \, ds +  \int_0^t \left( D_{v}+ D_{u_1} + D_{u_2} \right)\, ds  \\ 
		& \le C_0 \left(\norm{v_0-\tv}_{H^1(\mathbb{R})}^2 +\norm{u_0-\tu}_{H^1(\mathbb{R})}^2\right),
		\end{aligned}
		\end{equation}
		
		where $C_0$ is independent of $T$, and
		\begin{equation}\label{good terms}
		\begin{aligned}
		&G_1:=\int_\R |a_x^X|\left|p(v)-p(\tv^X)-\frac{u-\tu^X}{2C_*}\right|^2\,dx, \quad G^S:=\int_\R |\tv_x^X||u-\tu^X|^2\,dx,\\
        &D_v:=\int_\mathbb{R} |(v-\tv^X)_x|^2 \, dx, \quad D_{u_1}:=\int_\R |(u-\tu^X)_x|^2\,dx,\quad D_{u_2}:=\int_\R |(u-\tu^X)_{xx}|^2\,dx.
		\end{aligned}
		\end{equation}
		Here, $C_*$ is a positive constant defined in \eqref{C_star}.
	\end{proposition}

	\subsection{Proof of Theorem \ref{thm:main}}

Based on Propositions \ref{prop:local} and \ref{apriori-estimate}, we use the continuation argument to prove  the global-in-time existence of perturbations. We also use Proposition \ref{apriori-estimate} to prove the long-time behavior. Since the proof is similar to that of the previous articles (e.g. \cite{HKK23,KVW23}), we present it in Appendix, and complete the proof of Theorem \ref{thm:main}.\\

	In the following sections, we will prove Proposition \ref{apriori-estimate} by using the $a$-contraction method, without using the $h$ variable as mentioned in Introduction.  In Section \ref{sec:rel_ent}, we provide the estimate on the relative entropy between the solution to the NS equation and the viscous shock with shifts, which gives the $L^2$-estimate for $(v,u)$ perturbations. Then, we obtain $H^1$-estimate for $(v,u)$ in Section \ref{sec:high-order}.
	
\section{Estimate on the weighted relative entropy with the shift}\label{sec:rel_ent}
	\setcounter{equation}{0}
 In this section, we estimate the $L^2$-perturbation of a solution to the NS equations \eqref{eq:NS} from the viscous shock profile \eqref{viscous-shock} by using the method of $a$-contraction. In what follows,  we omit the dependency on the shift $X$ in \eqref{viscous-shock} and \eqref{a} for the simplicity of notation as follows:
		\begin{equation*}
		(\tv,\tu)(t,x):=(\tv,\tu)(x-\sigma t -X(t)),\quad a(t,x):=a(x-\sigma t - X(t)).
		\end{equation*}

	The main goal of this section is to verify the following control on the $L^2$-perturbation between the solution $(v,u)$ to \eqref{eq:NS} and the viscous shock $(\widetilde{v},\widetilde{u})$ defined as in  \eqref{viscous-shock}.
	\begin{lemma} \label{Main Lemma} There exists a positive constant $C$ such that for all $t \in [0,T],$
		\begin{equation} \label{energy-est}
		\begin{aligned}
		&\int_\mathbb{R} \big( |u-\tu|^2 +|v-\tv|^2   \big) \, d x + \int_0^t \left( \delta  | \dot{X}|^2 + G_1+G^S+D \right) \, ds \\ 
		&\quad \le C \int_\mathbb{R} \big( |u_0-\tu|^2 +|v_0- \tv|^2  \big) \, dx,
		\end{aligned}
		\end{equation}
		where $G_1$, $G^S$, and $D_{u_1}$ are defined in \eqref{good terms}.
	\end{lemma}

	\subsection{Relative entropy method}
	To prove Lemma \ref{Main Lemma}, we rely on the relative entropy method. To this end, we rewrite the NS system \eqref{eq:NS} in the following abstract form:
	\begin{equation}\label{eq:NS-abs}
	\partial_t U +\partial_x A(U) = \partial_x (M(U) \partial_x D \eta (U)),
	\end{equation}
	where 
	\[U:=\begin{pmatrix}
	v\\u
	\end{pmatrix},\quad 
	A(U):=\begin{pmatrix}
	-u\\p(v) 
	\end{pmatrix},\quad 
	D\eta(U):=\begin{pmatrix}
	-p(v)\\u ,
	\end{pmatrix},\]
	and 
	\begin{equation*}
	M(U):=
	\begin{pmatrix}
	0 &0 \\
	0 &\frac{1}{v} \\
	\end{pmatrix}.
	\end{equation*}
	 Similarly, the shifted viscous wave $\tU$ 
	\begin{equation} \label{shifted underlying wave}
	\tU(t,x):=\widetilde{U}(x-\sigma t-X(t)),
	\end{equation}
	satisfies a similar system as
	\[\partial_t\tU +\partial_x A(\tU)=\partial_x \left(M(\tU) \partial_x D \eta(\tU) \right) -\dot{X} \partial_x \tU.\]
	
	Now, we define the relative entropy between $U=(v,u)$ and $\overline{U}=(\overline{v},\overline{u})$ as
	\[\eta(U|\overline{U}):=\eta(U)-\eta(\overline{U})-D\eta(\overline{U})(U-\overline{U}),\]
	and the relative flux $A(U|\overline{U})$ 
	\[A(U|\overline{U})=A(U)-A(\overline{U})-DA(\overline{U})(U-\overline{U}).\]
	Finally, let  $G(U;\overline{U})$ be the relative entropy flux definded as
	\[G(U;\overline{U})=G(U)-G(\overline{U})-D\eta(\overline{U})(A(U)-A(\overline{U})),\]
	where $G$ is the entropy flux for $\eta$ satisfying the condition $D_iG(U)=\sum_{k=1}^2 D_k \eta(U)D_iA_k(U)$. In the case of NS system, we consider $G(U)=p(v)u$, and therefore, we can compute $A(U|\overline{U})$ and $G(U;\overline{U})$ as
	\[\eta(U|\overline{U})=\frac{|u-\overline{u}|^2}{2}+Q(v|\overline{v}),\quad A(U|\overline{U})=\begin{pmatrix}
	0\\
	p(v|\overline{v})
	\end{pmatrix}, \quad \text{and} \quad 
	G(U;\overline{U})=(p(v)-p(\overline{v}))(u-\overline{u}).\]
	Here, the relative internal energy $Q(v|\overline{v})$ and the relative pressure $p(v|\overline{v})$ are defined as
	\[Q(v|\overline{v})=Q(v)-Q(\overline{v})-Q'(\overline{v})(v-\overline{v}) \quad \mbox{and} \quad p(v|\overline{v})=p(v)-p(\overline{v})-p'(\overline{v})(v-\overline{v}).\]

	In order to obtain $L^2$-perturbation estimate in Lemma \ref{Main Lemma}, we focus on estimating the weighted relative entropy between the solution $U$ and the shifted viscous shock wave $\tU$:
	\[\int_\mathbb{R} a (t,x) \eta\left(U(t,x)|\tU(t,x)\right)\,d x,\quad\mbox{where}\quad a(t,x):=a(t,x-X(t)).\]
	
	\begin{lemma} \label{lem:rel-ent}
		Let $a$ be the weight function defined by \eqref{a} and $X:[0,T]\to\bbr$ be any Lipschitz continuous function. Let $U$ be a solution to \eqref{eq:NS-abs},  and $\tU$ be the shifted viscous shock wave defined in  \eqref{shifted underlying wave}.  Then
		\begin{equation}\label{est-weight-rel-ent} 
		\frac{d}{dt}\int_\mathbb{R} a (t,x) \eta (U(t,x))|\tU(t,x)) \, dx=\dot{X}(t)Y(U)+\mathcal{J}^{\textup{bad}}(U)-\mathcal{J}^{\textup{good}}(U),
		\end{equation}
		where the terms $Y(U)$, $\mathcal{J}^{\textup{bad}}$, and $\mathcal{J}^{\textup{good}}$ are defined as
		\begin{align*}
		Y(U)&:=-\int_\mathbb{R} a_x \eta(U|\tU)\,d x +\int_\mathbb{R} a D^2\eta(\tU)\widetilde{U}_x^X (U-\tU)\,d x,\\
		\mathcal{J}^{\textup{bad}}(U)&:=\int_\mathbb{R} a_x (p(v)-p(\tv))(u-\tu) \, dx-\int_\mathbb{R} a \tu_x^X p(v | \tv) \,d x\\
		&\quad-\int_\mathbb{R} a_x \frac{(u-\tu) \partial_x (u-\tu)}{v}   dx + \int_\mathbb{R} a_x   \frac{(u-\tu)(v-\tv) \partial_x \tu}{v \tv}  dx+\int_\mathbb{R} a \partial_x (u-\tu)  \frac{(v-\tv) }{v \tv} \partial_x \tu \, dx,\\
		\mathcal{J}^{\textup{good}}(U)&:= \frac{\sigma}{2}\int_\mathbb{R} a_x |u-\tu|^2 \, d x+\sigma \int_\mathbb{R} a_x Q(v|\tv) \, dx + \int_\mathbb{R} a \frac{|\partial_x (u-\tu)|^2}{v} \,dx.
		\end{align*}
	\end{lemma}
	\begin{proof}
		Since the system \eqref{eq:NS-abs} is written in a general hyperbolic system, we may use the same computations as in \cite[Lemma 2.3]{KV21} to estimate the time derivative of the weighted relative entropy as
		\begin{align*}
		\frac{d}{dt} \int_\mathbb{R} a \eta ( U | \tU) \, dx &=\dot{X}(t)Y(U)-\sigma\int_\mathbb{R} a_x\eta(U|\tU) \, dx+ \sum_{i=1}^5 I_{1i},
		\end{align*}
		where
		
		\begin{align*}
		I_{11}&:=-\int_\R a \partial_x G(U;\tU) \, dx=\int_\R a_x (p(v)-p(\tv) )(u-\tu)\, dx,\\
		I_{12}&:=-\int_\R a \partial_x D \eta (\tU) A(U| \tU) \, dx=-\int_\R a \tu_x^X p(v|\tv) \, dx,\\
		I_{13}&:=\int_\R a \left(D\eta(U)-D\eta(\tU) \right) \partial_x \left( M(U)\partial_x \left(D\eta (U)-D\eta(\tU) \right)\right) \, dx\\
  & \ =\int_\R a  (u-\tu) \partial_x \left( \frac{\partial_x (u-\tu)}{v} \right)  dx\\
		&\ = -\int_\R a \frac{|\partial_x(u-\tu)|^2}{v} \, dx  -\int_\R a_x  (u-\tu) \frac{\partial_x(u-\tu)}{v}  dx,\\
		I_{14}&:=\int_\R a \left( D \eta(U)-D\eta(\tU)\right)\partial_x \left( (M(U) -M(\tU)) \partial_x D \eta(\tU)\right)\,dx\\
  &=\int_\R a  (u-\tu) \partial_x \left( \left(\frac{1}{v}-\frac{1}{\tv} \right)\partial_x \tu \right) dx\\
		&=\int_\R a_x \frac{(v-\tv)(u-\tu)\partial_x \tu}{v \tv}  \, dx  + \int_\R a \partial_x (u-\tu)  \frac{(v-\tv)\partial_x \tu}{v \tv} \, dx,\\
		I_{15}&:=\int_\R a (D\eta)(U|\tU)\partial_x \left( M(\tU) \partial_x D \eta(\tU) \right)\, dx=0.
		\end{align*}
		Thus, combining all the estimates on $I_{1i}$ above, we obtain the desired estimate.
	\end{proof}

	\subsection{Maximization on $p(v)-p(\tv)$} 
	Among the terms in $\mathcal{J}^\text{bad}$, a primary bad term is 
	\[\int_\mathbb{R} a_x (p(v)-p(\tv))(u-\tu) \, d x\]
	where the perturbations for $p(v)$ and $u$ are coupled. In order to exploit the parabolic term on $u$-variable and hence use the Poincar\'e-type inequality, we separate $u-\tu$ from $p(v)-p(\tv)$ by using the quadratic structure of $p(v)-p(\tv)$. We first obtain the following estimates on several terms in $\mathcal{J}^{\textup{bad}}(U)$ and $\mathcal{J}^{\textup{good}}(U)$. 

	\begin{lemma} \cite[Lemma 4.3]{HKKL}\label{lem:quad} For any $\delta>0$ small enough, let $C_*$ be the constant as
		 \begin{equation}\label{C_star}
		C_* = \frac{1}{2}\left(\frac{1}{\sigma_{\ell}}-\sqrt{\delta }\frac{\gamma+1}{\gamma}\frac{1}{p(v_-)}\right), \quad \sigma_\ell =\sqrt{-p'(v_-)}.
		\end{equation}
  Then, we have
		\begin{align}
		\begin{aligned}\label{quadratic estimate}
		&-\int_{\mathbb{R}} a \tu_x p(v|\tv) \, dx-\sigma \int_{\mathbb{R}} a_x Q(v|\tv) \, dx \\ &\le -C_*\int_{\R}a_x|p(v)-p(\tv)|^2\,dx+C\delta  \int_\R a_x\big|p(v)-p(\tv)\big|^2 \, d x+C\int_\R a_x\big|p(v)-p(\tv)\big|^3 \, d x.
		\end{aligned}
		\end{align}
	\end{lemma}
 Using Lemma \ref{lem:quad} and quadratic structure with respect to  $p(v)-p(\tv)$, we can derive an upper bound of $\mathcal{J}^{\textup{bad}} -\mathcal{J}^{\textup{good}}(U)$.
 \begin{lemma} For $\mathcal{J}^{\textup{bad}}$ and $\mathcal{J}^{\textup{good}}(U)$ defined in \eqref{est-weight-rel-ent}, we have
   	\[\mathcal{J}^{\text{bad}} - \mathcal{J}^{\text{good}} \le  \mathcal{B}-\mathcal{G}.\]
    Here, $\mathcal{B}$ and $\mathcal{G}$ are defined as follows:
	\begin{align*}
	&\mathcal{B}(U) =\sum_{i=1}^6 \mathcal{B}_i(U),\\
	&\mathcal{G}(U) = \mathcal{G}_1(U) + \mathcal{G}_2(U) + \mathcal{D}(U).
    \end{align*}
	where $\mathcal{B}_i$, $\mathcal{G}_i$, and $\mathcal{D}$ are given by:
    \begin{align*}
	&\mathcal{B}_1 :=\frac{1}{4 C_*} \int_\mathbb{R} a_x  |u-\tu|^2  dx, \quad \mathcal{B}_2 :=-\int_\mathbb{R} a_x \frac{(u-\tu) \partial_x (u-\tu)}{v}  dx, \quad \mathcal{B}_3 := \int_\mathbb{R} a_x \frac{(u-\tu)(v-\tv) \partial_x \tu}{v \tv} dx, \\
  &\mathcal{B}_4 :=\int_\mathbb{R} a   \frac{(v-\tv)\partial_x (u-\tu)\partial_x \tu}{v \tv}  dx, \quad \mathcal{B}_5 :=C \delta  \int_\R a_x\big|p(v)-p(\tv)\big|^2 d x, \quad \mathcal{B}_6 :=C\int_\R a_x\big|p(v)-p(\tv)\big|^3  d x,
	\end{align*}
    and
	\begin{align*}
	&\mathcal{G}_1:=C_* \int_\mathbb{R} a_x \left|p(v)-p(\tv)-\frac{u-\tu}{2C_*} \right|^2 \, d x,\quad \mathcal{G}_2:=\frac{\sigma}{2} \int_\mathbb{R} a_x |u-\tu|^2  \, dx,\quad \mathcal{D}:= \int_\mathbb{R} a \frac{|\partial_x (u-\tu)|^2}{v} \, dx.
    \end{align*}
 \end{lemma}
 \begin{proof}
    We use Lemma \ref{lem:quad} to derive a quadratic structure in terms of $p(v)-p(\tv)$ as follows:
    \begin{equation}\label{eq: max}
          \begin{aligned} 
        &\int_\mathbb{R} a_x (p(v)-p(\tv))(u-\tu) \,dx -\int_\mathbb{R} a \tu_x p(v|\tv) \, dx -\sigma \int_\mathbb{R} a_x Q(v|\tv) \, dx  \\ 
        &\le \int_\mathbb{R} a_x (p(v)-p(\tv))(u-\tu) \,dx  -C_*\int_{\R}a_x|p(v)-p(\tv)|^2\,dx\\ 
        &\quad +C\delta  \int_\R a_x\big|p(v)-p(\tv)\big|^2 \, d x+C\int_\R a_x\big|p(v)-p(\tv)\big|^3 \, d x.
    \end{aligned}
    \end{equation}
    Then, the first two terms in the right-hand side of \eqref{eq: max} can be rewritten by using the quadratic structure of $p(v)-p(\tv)$ as 
    \begin{equation}
        \begin{aligned}
            &\int_\mathbb{R} a_x (p(v)-p(\tv))(u-\tu) \,dx  -C_*\int_{\R}a_x|p(v)-p(\tv)|^2\,dx \\&=\int_\mathbb{R} a_x \left(-C_* \left( (p(v)-p(\tv))-\frac{u-\tu}{2C_*}\right)^2+\frac{(u-\tu)^2}{4C_*} \right) dx.
        \end{aligned}
    \end{equation}
    Since the other terms in $\mathcal{J^{\text{bad}}}$ and $\mathcal{J}^{\text{good}}$ are unchanged, we obtain the desired inequality.
 \end{proof}
 The estimate \eqref{est-weight-rel-ent} in Lemma \ref{lem:rel-ent} can be further bounded as
	\begin{equation}\label{est-1}
	\frac{d}{dt}\int_{\bbr}a\eta(U|\tU)\,dx \le \dot{X}(t) Y(U) + \mathcal{B}(U)-\mathcal{G}(U),
	\end{equation}
 On the other hand, since $Y(U)$ is expanded as
	\begin{align*}
	Y(U)&=-\int_{\bbr} a_x \eta(U|\tU)\,dx +\int_\bbr a D^2\eta(\tU)(\widetilde{U})^X_x (U-\tU)\,dx\\
	& = -\int_{\bbr} a_x \left(\frac{|u-\widetilde{u}^X|^2}{2}+Q(v|\widetilde{v}^X)\right)\,dx\\
	&\quad +\int_{\bbr} a  \tu_x (u-\widetilde{u}^X)  \, dx
	 -\int_{\bbr} a p'(\widetilde{v}^X)\widetilde{v}^X_x (v-\widetilde{v}^X)\,dx,
	\end{align*}
	we decompose $Y$ as
	\[Y= \sum_{i=1}^4 Y_{i},\]
	where
	\begin{align*}
	Y_{1} &:=\int_{\mathbb{R}} a \tu_x (u-\tu) \,dx,& Y_{2}&:=\frac{1}{\sigma}\int_{\mathbb{R}} a p'(\tv)\tv_x (u-\tu) \,dx,\\
	Y_{3} &:=- \int_\mathbb{R} a p'(\tv) \tv_x \left(v-\tv+\frac{(u-\tu)}{\sigma} \right) \, d x, &Y_{4} &:= -\int_{\bbr} a_x \left(\frac{|u-\widetilde{u}^X|^2}{2}+Q(v|\widetilde{v}^X)\right)\,dx	
	\end{align*}
	We now define a shift function $X(t)$ so that it satisfies the following ODE:  
	\begin{equation}\label{shift}
	\dot{X} = -\frac{M}{\delta }(Y_{1}+Y_{2}),\quad X(0)=0.
	\end{equation}
	With this choice of shift $X$, the term $\dot{X}(t)Y(U)$ in \eqref{est-1} can be written as
	\[\dot{X}(t)Y(U)= -\frac{\delta }{M}|\dot{X}|^2+\dot{X}\sum_{i=3}^4Y_{i}.\]
	To summarize, we decompose the right-hand side of \eqref{est-1} as
	\begin{align}
	\begin{aligned}\label{est}
	\frac{d}{dt}\int_{\bbr}a\eta(U|\tU)\,dx &= \underbrace{-\frac{\delta }{2M}|\dot{X}|^2 + \mathcal{B}_1 -\mathcal{G}_2-\frac{3}{4}\mathcal{D}}_{=: \mathcal{R}_1}\\
	&\quad \underbrace{-\frac{\delta }{2M}|\dot{X}|^2 + \dot{X}\sum_{i=3}^4 Y_{i} +\sum_{i=2}^6\mathcal{B}_i -\mathcal{G}_1-\frac{1}{4}\mathcal{D}}_{=:\mathcal{R}_2}.
	\end{aligned}
	\end{align}
	In the following subsections, we estimate the terms in $\mathcal{R}_1$ and $\mathcal{R}_2$ respectively.
	
	\subsection{Estimate of $\mathcal{R}_1$} \label{Est-main-part}
	Estimation of $\mathcal{R}_1$ is the most important part in the proof of Lemma \ref{Main Lemma}, in which the Poincar\'e-type inequality is crucially used. For a fixed $t\ge0$, we define an auxiliary variable $y$ as
	\begin{equation*} 
	y:=\frac{u_- -\tu(x-\sigma t-X(t))}{\delta }.
	\end{equation*}
	Then it follows from the definition that the map $x\mapsto y=y(x)$ is one-to-one and 
	\begin{equation*}
	\frac{d y}{d x} = -\frac{1}{\delta } \tu_x>0,\quad\mbox{and}\quad\lim_{x\to -\infty} y=0, \quad \lim_{x \to \infty}y=1.
	\end{equation*}
	Using the new variable $y$, we will apply the following Poincarè-type inequality:
	
	\begin{lemma}\cite[Lemma 2.9]{KV21}\label{lem: KV inequality} For any $f:[0,1] \to \R$ with $\int_0^1 y(1-y) |f'|^2 \, dy <\infty$, 
		\begin{equation*}
		\int_0^1 \left| f-\int_0^1 f \,dy \right|^2 dy \le \frac{1}{2} \int_0^1 y(1-y)|f'|^2 \,dy.
		\end{equation*}
	\end{lemma}
	
	We apply Lemma \ref{lem: KV inequality} to the perturbation $f$ of the following form:
	\begin{equation*} 
	f:=\left(u(t,\cdot)-\tu(\cdot-\sigma t-X(t)\right)\circ y^{-1}.
	\end{equation*} Therefore, the goal of this subsection is to represent $\mathcal{R}_1$ in terms of $f$ and then use Lemma \ref{lem: KV inequality} to estimate it. In the following, we estimate the terms in $\mathcal{R}_1$ separately.\\
	
	\noindent $\bullet$  (Estimate of $\frac{\delta }{2M}|\dot{X}|^2$): 
	We use the definition of $Y_1$ and change of variables for $y$ to observe that
	\[Y_1=\int_{\R} a \tu_x^X (u-\tu)\,dx
	=-\delta  \int_0^1 a f \,dy.\]
	Then, using \eqref{shock_speed_est} and $\| a-1 \|_{L^\infty(\R_+ \times \R)} \le \sqrt{\delta }$, we have
	\[\left| Y_1+\delta  \int_0^1 f \, dy \right|\le\delta \int_0^1|a-1||f| dy \le   \delta ^{3/2} \int_0^1 |f| \, dy.\]
	To estimate $Y_2$, we first use the relation $\sigma \tv_x=-\tu_x $ and change of variables for $y$ to yield
	\[Y_2=-\frac{1}{\sigma^2} \int_{\mathbb{R}} a p'(\tv)\tu_x (u-\tu) \, d x=\frac{\delta }{\sigma^2} \int_0^1 a p'(\tv) f \, dy.\]
	This, together with the estimates \eqref{shock_speed_est}, \eqref{shock_speed_est-2}, and $\| a-1 \|_{L^\infty(\R_+ \times \R)} \le \sqrt{\delta }$, implies 
	\begin{align*}
	\left| Y_2 + \delta  \int_0^1 f \, dy \right|\le \delta \int_0^1\left|\frac{a p'(\tv)}{\sigma^2}+1\right||f|dy \le C \delta  (\sqrt{\delta }+\delta ) \int_0^1 |f| \, dy.
	\end{align*}
	Since $\dot{X}=-\frac{M}{\delta }(Y_1+Y_2)$, we combine the estimates for $Y_1$ and $Y_2$ to obtain
	\begin{align*}
	\left| \dot{X} - 2 M \int_0^1 f \, dy \right| \le \frac{M}{\delta } \left( \left| Y_1 +\delta  \int_0^1 f \, dy \right|+\left| Y_2+\delta  \int_0^1 f\, dy \right| \right)\le C (\sqrt{\delta }+ \delta ) \int_0^1 |f| \, dy,
	\end{align*}
	which implies
	\begin{align*}
	\left( \left| 2M \int_0^1 f \, dy \right| -|\dot{X}| \right)^2 \le C(\sqrt{\delta }+\delta )^2 \int_0^1 |f|^2 \, dy\le C\delta \int_0^1|f|^2\,d y.
	\end{align*}
	We use an elementary inequality $\frac{p^2}{2}-q^2 \le (p-q)^2$ for $p,q \in \mathbb{R}$ to obtain 
	\begin{align*}
	2M^2\left(\int_0^1 f \, dy \right)^2-|\dot{X}|^2 \le C \delta  \int_0^1 |f|^2 \, dy,
	\end{align*}
	and by rearranging the terms, we obtain
	\begin{equation}\label{est-dotX}
	-\frac{\delta }{2M} |\dot{X}|^2 \le -M \delta  \left(\int_0^1 f \, dy \right)^2 +C \delta  \int_0^1 |f|^2 \, dy.
	\end{equation}
	
	\noindent $\bullet$ (Estimates of $\mathcal{B}_1$ and  $\mathcal{G}_2$):
	Recall that $\mathcal{B}_1$ and $\mathcal{G}_2$ are 
	\begin{align*}
	&\mathcal{B}_1:=\frac{1}{4C_*}\int_\mathbb{R} a_x |u-\tu|^2 \, d x \ \  \text{and} \ \ \mathcal{G}_2:=\frac{\sigma}{2} \int_\mathbb{R} a_x |u-\tu|^2 \, d x.
	\end{align*}
	Therefore,
	\begin{align*}
	\mathcal{B}_1-\mathcal{G}_2&=\left(\frac{1}{4C_*}-\frac{\sigma}{2}\right)\int_{\R}a_x|u-\tu|^2\,dx=-\left(\frac{1}{4C_*}-\frac{\sigma}{2}\right)\frac{1}{\sqrt{\delta }}\int_{\R}\tu_x^X|u-\tu|^2\,dx\\
	&=\sqrt{\delta }\left(\frac{1}{4C_*}-\frac{\sigma}{2}\right)\int_0^1 f^2\,dy.
	\end{align*}
	where $C_*$ defined in \eqref{C_star} can be written as
	\begin{equation*}
	C_*= \frac{1}{2 \sigma_\ell} -(\sqrt{\delta }+\delta )\alpha_\ell \sigma_\ell,
	\end{equation*}
	with simple notation $\alpha_\ell=\frac{\gamma+1}{2 \gamma \sigma_\ell p(v_-)}$ in \eqref{O(1) constnats}. On the other hand, using \eqref{shock_speed_est}, \eqref{shock_speed_est-2}, we obtain
	\begin{align*}
	\sqrt{\delta }\left(\frac{1}{4C_*}-\frac{\sigma}{2}\right)&\le \frac{\sigma_\ell}{2} \frac{\sqrt{\delta }}{1-2 (\sqrt{\delta }+\delta ) \sigma_\ell^2 \alpha_\ell}- \frac{\sigma }{2} \sqrt{\delta } 
	\\
	&\le \frac{\sqrt{\delta }}{2} \left( \frac{1}{1-2 (\sqrt{\delta }+\delta ) \sigma_\ell^2 \alpha_\ell} (\sigma_\ell -\sigma) +\sigma \left(\frac{1}{1-2 (\sqrt{\delta }+\delta ) \sigma_\ell^2 \alpha_\ell}-1\right) \right) \\
	&\le C  \delta ^{3/2} +\delta \sigma_\ell^3 \alpha_\ell.
	\end{align*}

	Therefore, we have
	\begin{equation}\label{R1.2}
	\mathcal{B}_1-\mathcal{G}_2 \le   C\delta ^{3/2} \int_0^1 f^2 \, dy+  \sigma_\ell^3 \alpha_\ell \delta  \int_0^1 f^2 \, dy. 
	\end{equation}

	\noindent $\bullet$ (Estimate of $\mathcal{D}(U)$):
	First, using $a \geq 1$ and change of variables, we estimate the diffusion term $\mathcal{D}$ in terms of $f$:
	\begin{equation*}
	\mathcal{D} \geq \int_\R \frac{1}{v}|\partial_x (u-\tu)|^2 dx=\int_0^1 |\partial_y f|^2 \frac{1}{v} \left(\frac{dy}{d x} \right) dy.
	\end{equation*}
	
	Similar to the previous result in \cite[Lemma 4.5]{KVW23}, there exists $C>0$ such that the following estimate holds:
	\begin{equation}\label{Diffusion}
	\left|   \frac{1}{y(1-y)} \frac{1}{\tv} \left( \frac{dy}{dx} \right)-\frac{\sigma}{2 \sigma_\ell}\frac{\delta  v''(p_-)}{ |v'(p_-)|^2 } \right| \le C \delta ^2.
	\end{equation}
	 On the other hand, since $C^{-1} \leq v \leq C$, we have
	\begin{equation}\label{estimate in diffusion}
	\left| \frac{\tv}{v}-1 \right| \leq C \left|\tv-v \right| 
	\leq C \varepsilon.
	\end{equation}
	Then, using \eqref{Diffusion} and \eqref{estimate in diffusion}, we obtain the lower bound for $\mathcal{D}$ as
	\begin{align*}
	\mathcal{D} &\geq \int_0^1 |\partial_y f|^2 \frac{\tv}{v} \frac{1}{\tv} \left(\frac{dy}{dx} \right) dy \\
	&\geq (1-C \varepsilon) \left(\frac{\sigma}{2 \sigma_\ell}\frac{\delta  v''(p_-)}{ |v'(p_-)|^2 }-C \delta ^2 \right) \int_0^1 y(1-y) \left| \partial_y f \right|^2 dy.
	\end{align*}
	Finally, since
	\begin{align*}
	\sigma_\ell^3 \alpha_\ell=\frac{1}{2} (1+\gamma) \frac{1}{v_-}=\frac{1}{2 }\frac{v''(p_-)}{|v'(p_-)|^2 },
	\end{align*}
	we obtain
	\begin{equation} \label{R1.3}
	\mathcal{D} \ge \sigma_\ell^3 \alpha_\ell \delta  (1-C (\delta_0 + \varepsilon)) \int_0^1 y(1-y) |\partial_y f|^2 \, dy.
	\end{equation}
	
	\noindent $\bullet$ (Estimate of $\mathcal{R}_1$): We now combine the estimates \eqref{R1.2} and \eqref{R1.3}, we have
	\begin{align*}
	\mathcal{B}_1-\mathcal{G}_2-\frac{3}{4}\mathcal{D} \le \sigma_\ell^3 \alpha_\ell \delta  \left(   (1+C \sqrt{\delta })\int_0^1 f^2 \, dy - \frac{3}{4}(1-C(\delta_0+\varepsilon)) \int_0^1 y(1-y) |\partial_y f|^2 \, dy \right),
	\end{align*}
	which together with the smallness of $\delta_0,\varepsilon$ yields
	\[\mathcal{B}_1-\mathcal{G}_2-\frac{3}{4}\mathcal{D} \le \sigma_\ell^3 \alpha_\ell \delta  \left(   \frac{9}{8} \int_0^1 f^2 \, dy -\frac{5}{8}  \int_0^1 y(1-y) |\partial_y f|^2 \, dy \right). \] 
	Then, using Lemma \ref{lem: KV inequality}  with the identity
	\[\int_0^1 |f-\overline{f}|^2 \, dy=\int_0^1 f^2 \, dy-\overline{f}^2, \quad \overline{f}:=\int_0^1 f \, dy,\]
	we have
	\begin{align*}
	\mathcal{B}_1-\mathcal{G}_2-\frac{3}{4}\mathcal{D} \le -\frac{\sigma_\ell^3 \alpha_\ell \delta }{8} \int_0^1 f^2 \, dy + \frac{5 \sigma_\ell^3 \alpha_\ell \delta }{4} \left(\int_0^1 f\,dy\right)^2.
	\end{align*}
	Finally, using \eqref{est-dotX} with the choice $M=\dfrac{5 \sigma_\ell^3 \alpha_\ell}{4}$, we have
	\begin{align*}
	&-\frac{\delta }{2M}|\dot{X}|^2+\mathcal{B}_1-\mathcal{G}_2-\frac{3}{4}\mathcal{D} \le  -\frac{\sigma_\ell^3 \alpha_\ell \delta }{16} \int_0^1 f^2 \, dy,
	\end{align*}
	which implies 
	\begin{equation} \label{R1}
	\begin{aligned}
	\mathcal{R}_1 \le -C_1 \int_\mathbb{R} |\tv_x| |u-\tu|^2 \, dx=:-C_1G^S.
	\end{aligned}
	\end{equation}
	\subsection{Estimate of remaining terms} \label{Est-Remaining-terms}
	We now estimate the remaining terms in $\mathcal{R}_2$. We first substitute the estimate of $\mathcal{R}_1$ \eqref{R1} to \eqref{est} and use the Young's inequality
	\begin{equation*}
	\dot{X}\sum_{i=3}^4 Y_i \le \frac{\delta }{4M} |\dot{X}|^2 +\frac{C}{\delta }\sum_{i=3}^4 |Y_i|^2
	\end{equation*}
	to have
	\begin{equation}\label{est-R}
	\begin{aligned}{}
	\frac{d}{dt} \int_\mathbb{R} a \eta(U|\tU) \, dx \le -C_1 G^S -\frac{\delta }{4M} |\dot{X}|^2 +\frac{C}{\delta }\sum_{i=3}^4 |Y_i|^2 + \sum_{i=2}^6 \mathcal{B}_i -\mathcal{G}_1  -\frac{1}{4} \mathcal{D}.
	\end{aligned}
	\end{equation}
	Therefore, to close the estimate of the weighted relative entropy, it suffices to control the remaining terms $|Y_i|^2$ and $\mathcal{B}_i$.\\
	\begin{lemma} \label{est-R2}For sufficiently small $\delta$ and $\varepsilon$
 	\begin{equation}
	\begin{aligned}
	\frac{C}{\delta }\sum_{i=3}^4 |Y_i|^2 \le \frac{1}{32} (\mathcal{G}_1  + C_1 G^S), \quad \sum_{i=2}^6 \mathcal{B}_i  \le \frac{5}{32} \left( \mathcal{G}_1 + C_1 G^S + \mathcal{D} \right). 
	\end{aligned}
	\end{equation}

	\end{lemma}
	\begin{proof}
	    Since the proof is essentially the same as in \cite[Section 4.5]{HKKL}, we omit the proof.
	\end{proof}

	\subsection{Proof of Lemma \ref{Main Lemma}}
	We are now ready to prove the key lemma, Lemma \ref{Main Lemma}. We combine all the estimates in \eqref{est-R} and Lemma \ref{est-R2} to derive the following control on the weighted relative entropy:
	\begin{align*}
	\frac{d}{dt} \int_\mathbb{R} a \eta (U | \tU ) \, dx &\le -\frac{\delta }{4 M} |\dot{X}|^2 -\frac{11}{16} \mathcal{G}_1   -\frac{10}{16}C_1 G^S-\frac{1}{16} \mathcal{D}.
	\end{align*}
	After integrating the above inequality on $[0,t]$ for any $t \le T$, we conclude that
	\begin{align*}
	&\int_\mathbb{R} a(t,x) \eta (U (t,x) | \tU (t,x) ) \, dx + \int_0^t \left( \delta  |\dot{X}|^2 + \mathcal{G}_1 + G^S+ \mathcal{D} \right) \, ds \\ 
	& \quad \le C  \int_\mathbb{R} a(0,x) \eta ( U_0(x)| \tU(0,x)) \, dx.
	\end{align*}
	However, since the bounds $1/2 \le a \le 2$, $D(U) \le C \mathcal{D}(U)$, $G_1(U) \le C \mathcal{G}_1(U)$, and the relation
	\begin{equation*}
	\|U-\tU\|_{L^2(\mathbb{R})}^2 \sim \int_\mathbb{R} \eta(U|\tU) \, dx, \quad \forall t \in [0,T]
	\end{equation*} 
	holds, this completes the proof of Lemma \ref{Main Lemma}.\\

	\section{Estimate on the $H^1$-perturbation}\label{sec:high-order}
	\setcounter{equation}{0}		
\begin{lemma}
There exists a constant $C>0$ such that
\begin{align}
\begin{aligned} \label{lem5.1}
&\lVert v-\widetilde{v} \rVert^2_{H^1(\mathbb{R})}+\lVert u-\widetilde{u} \rVert^2_{L^2(\mathbb{R})}+\int_{0}^{t}D_{v} d\tau+\int_{0}^{t}\left(\delta|\dot{X}|^2+G_1+G^S+D_{u_1}\right)d\tau \\ 
\le&  C\left(\lVert (v-\widetilde{v})(0, \cdot) \rVert^2_{H^1(\mathbb{R})}+\lVert (u-\widetilde{u})(0, \cdot) \rVert^2_{L^2(\mathbb{R})}\right)+C\varepsilon\int_{0}^{t}D_{u_2}d\tau.
\end{aligned}
\end{align}
\end{lemma}
\begin{proof}
Recall
\begin{equation} \label{oeh0}
\begin{cases}
(v-\widetilde{v})_t=(u-\widetilde{u})_x+\dot{X}(t)\widetilde{v}_x & \\
(u-\widetilde{u})_t=-(p(v)-p(\widetilde{v}))_x+\left(\frac{u_x}{v}-\frac{\widetilde{u}_x}{\widetilde{v}}\right)_x+\dot{X}(t)\widetilde{u}_x. & 
\end{cases} 
\end{equation}
We differentiate \eqref{oeh0}$_1$ with respect to $x$ and multiply the result  by $(v - \widetilde{v})_x$ to obtain
\begin{align}
\begin{aligned} \label{5.2}
\frac{d}{dt}\frac{|(v-\widetilde{v})_x|^2}{2} &=(u-\widetilde{u})_{xx}(v-\widetilde{v})_x+ \dot{X}(t)\widetilde{v}_{xx}(v-\widetilde{v})_{x} .
\end{aligned}    
\end{align}
On the other hand, we multiply \eqref{oeh0}$_2$  by $-v(v - \widetilde{v})_x$ and use the following identity
\begin{equation} \label{p-tp eq}
(p(v)-p(\widetilde{v}))_x=p'(v)(v-\widetilde{v})_x-\widetilde{v}_x(p'(v)-p'(\widetilde{v}))   
\end{equation}
to get
\begin{align}
\begin{aligned} \label{5.3}
(-v(v-\widetilde{v})_x(u-\widetilde{u}))_t=&-\left(v(v-\widetilde{v})_t(u-\widetilde{u})\right)_x-v_t(v-\widetilde{v})_x(u-\widetilde{u})\\
&+v_x(v-\widetilde{v})_t(u-\widetilde{u})+v(v-\widetilde{v})_t(u-\widetilde{u})_x\\
&+vp'(v)|(v-\widetilde{v})_x|^2-v\widetilde{v}_x(v-\widetilde{v})_x\left(p'(v)-p'(\widetilde{v})\right)\\
&-v(v-\widetilde{v})_x\left(\frac{u_x}{v}-\frac{\widetilde{u}_x}{\widetilde{v}}\right)_x-v(v-\widetilde{v})_x\dot{X}\widetilde{u}_x.
\end{aligned}
\end{align}

Now, we combining \eqref{5.2} and \eqref{5.3} to have
\begin{equation} \label{5result1}
    \begin{aligned}
        &\frac{d}{dt}\left(\frac{|(v-\widetilde{v})_x|^2}{2}-v(v-\widetilde{v})_x(u-\widetilde{u})\right)+\left(v(v-\widetilde{v})_x(u-\widetilde{u})\right)_x-vp'(v)|(v-\widetilde{v})_x|^2 \\ =&\dot{X}(v-\widetilde{v})_x\left(\widetilde{v}_{xx}-v\widetilde{u}_x\right)-v_t(v-\widetilde{v})_x(u-\widetilde{u})+v_x(v-\widetilde{v})_t(u-\widetilde{u})\\
&+v(v-\widetilde{v})_t(u-\widetilde{u})_x-v\widetilde{v}_x(v-\widetilde{v})_x(p'(v)-p'(\widetilde{v}))+\frac{v_x(v-\widetilde{v})_x(u-\widetilde{u})_x}{v}\\
&+\frac{\widetilde{u}_{xx}(v-\widetilde{v})(v-\widetilde{v})_x}{\widetilde{v}}+v\widetilde{u}_x(v-\widetilde{v})_x\left(\frac{v_x}{v^2}-\frac{\widetilde{v}_x}{\widetilde{v}^2}\right),
    \end{aligned}
\end{equation}
where we used 
\[\left(\frac{u_x}{v}-\frac{\widetilde{u}_x}{\widetilde{v}}\right)_x=\frac{(u-\widetilde{u})_{xx}}{v}-\frac{(u-\widetilde{u})_xv_x}{v^2}+\widetilde{u}_{xx}\left(\frac{\widetilde{v}-v}{v\widetilde{v}}\right)-\widetilde{u}_x\left(\frac{v_x}{v^2}-\frac{\widetilde{v}_x}{\widetilde{v}^2}\right).\]

Integrating the equation \eqref{5result1} over $\left[0,t\right]\times \mathbb{R}$ with respect to $t$ and $x$, we have
\begin{align}
\begin{aligned} \label{5.4}
&\frac{\lVert (v-\widetilde{v})_x(t,\cdot) \rVert^2}{2}-\frac{\lVert (v-\widetilde{v})_x(0,\cdot) \rVert^2}{2}+\int_{0}^t\int_{\mathbb{R}}-vp'(v)|(v-\widetilde{v})_x|^2dxd\tau \\ 
=&\int_{\mathbb{R}}(v(v-\widetilde{v})_x(u-\widetilde{u}))(t,\cdot)dx-\int_{\mathbb{R}}(v(v-\widetilde{v})_x(u-\widetilde{u}))(0,\cdot)dx+\sum\limits_{i=1}^8\mathcal{L}_i, 
\end{aligned}
\end{align}
where 
\begin{align*} 
\mathcal{L}_1&=\int_{0}^{t}\int_{\mathbb{R}}\dot{X}(v-\widetilde{v})_x(\widetilde{v}_{xx}-v\widetilde{u}_x)dxd\tau, \ \mathcal{L}_2=-\int_{0}^{t}\int_{\mathbb{R}}v_t(v-\widetilde{v})_{x}(u-\widetilde{u})dxd\tau, \\
\mathcal{L}_3&=\int_{0}^{t}\int_{\mathbb{R}}v_x(v-\widetilde{v})_{t}(u-\widetilde{u})dxd\tau, \ \qquad  \quad \mathcal{L}_4=\int_{0}^{t}\int_{\mathbb{R}}v(v-\widetilde{v})_{t}\left(u-\widetilde{u}\right)_xdxd\tau, \\
\mathcal{L}_5&=-\int_{0}^{t}\int_{\mathbb{R}}v\widetilde{v}_x(v-\widetilde{v})_{x}\left(p'(v)-p'(\widetilde{v})\right)dxd\tau, \
\mathcal{L}_6=\int_{0}^{t}\int_{\mathbb{R}}\frac{v_x(v-\widetilde{v})_x(u-\widetilde{u})_x}{v}dxd\tau, \\
\mathcal{L}_7&=\int_{0}^{t}\int_{\mathbb{R}}\frac{\widetilde{u}_{xx}(v-\widetilde{v})(v-\widetilde{v})_x}{\widetilde{v}}dxd\tau, \
\quad \quad \mathcal{L}_8=\int_{0}^{t}\int_{\mathbb{R}}v\widetilde{u}_x(v-\widetilde{v})_x\left(\frac{v_x}{v^2}-\frac{\widetilde{v}_x}{\widetilde{v}^2}\right)dxd\tau.
\end{align*}  

In the above equation \eqref{5.4}, the first and second terms on the right-hand side can be estimated as follows.
\begin{align*}
&\int_{\mathbb{R}}(v(v-\widetilde{v})_x(u-\widetilde{u}))(t,\cdot)dx-\int_{\mathbb{R}}(v(v-\widetilde{v})_x(u-\widetilde{u}))(0,\cdot)dx\\
\le & C\left(\lVert (v-\widetilde{v})_x(t,\cdot) \rVert_{L^2}\lVert (u-\widetilde{u})(t,\cdot) \rVert_{L^2}+\lVert (v-\widetilde{v})_x(0,\cdot) \rVert_{L^2}\lVert (u-\widetilde{u})(0,\cdot) \rVert_{L^2}\right)\\
\le & \frac{1}{4}\lVert v|p'(v)|(v-\widetilde{v})_x(t,\cdot) \rVert^2_{L^2}+C\lVert (u-\widetilde{u})(t,\cdot) \rVert^2_{L^2}+C\lVert (v-\widetilde{v})_x(0,\cdot) \rVert^2_{L^2}+C\lVert (u-\widetilde{u})(0,\cdot) \rVert^2_{L^2}.
\end{align*}
Therefore, we can rewrite \eqref{5.4} as
\begin{align}
\begin{aligned} \label{5.7}
\lVert (v-\widetilde{v})_x\rVert^2_{L^2}+\int_{0}^{t}\lVert (v-\widetilde{v})_x \rVert^2_{L^2} d\tau \le& C\left(\lVert (v-\widetilde{v})_x(0,\cdot) \rVert^2_{L^2}+\lVert (u-\widetilde{u})(0,\cdot) \rVert^2_{L^2}\right)\\
&+C\lVert u-\widetilde{u}\rVert^2_{L^2}+C\sum\limits_{i=1}^8|\mathcal{L}_i|.  
\end{aligned}
\end{align}
As in Section 4, we can easily get the estimates of $L_i$'s as
\begin{align*}
\sum\limits_{i=1}^8|\mathcal{L}_i|\le \frac{3}{4}\int_{0}^{t}D_vd\tau+C\int_{0}^{t}D_{u_1}d\tau+C(\delta+\varepsilon)\int_{0}^{t}\left(\delta|\dot{X}|^2+G_1+G^{S}\right)+C\varepsilon\int_{0}^{t}D_{u_2}d\tau.
\end{align*}
Therefore, \eqref{5.7} yields the inequality below with some constant $C_2>0$
\begin{align}
\begin{aligned} \label{finequ2.}
&\lVert (v-\widetilde{v})_x\rVert^2_{L^2}+\int_{0}^{t}D_{v}  d\tau \\
\le& C\left(\lVert (v-\widetilde{v})_x(0,\cdot) \rVert_{L^2}^2+\lVert u_0-\widetilde{u}(0,\cdot) \rVert_{L^2}^2\right)+C_2\left(\lVert u-\widetilde{u}\rVert_{L^2}^2+\int_{0}^{t}D_{u_1}d\tau\right)\\
&+C(\varepsilon+\delta)\int_{0}^{t}\left(\delta|\dot{X}|^2+G_1+G^S\right)d\tau+C\varepsilon\int_{0}^{t}D_{u_2}d\tau. 
\end{aligned}    
\end{align}
Multiply \eqref{finequ2.} by $\frac{1}{\max \left(2, C_2\right)}$ and add to \eqref{Main Lemma}. Then the resultant inequality yields \eqref{lem5.1}.
\end{proof}

\begin{lemma} 
There exists a constant $C>0$ such that
\begin{align} 
\begin{aligned} \label{lem 5.2}
&\lVert v-\widetilde{v} \rVert^2_{H^1(\mathbb{R})}+\lVert u-\widetilde{u} \rVert^2_{H^1(\mathbb{R})}+\int_{0}^{t}\left(\delta|\dot{X}|
^2+{G}_1+G^S+D_{v}+D_{u_1}+D_{u_2}\right)d\tau \\  
\le &  C\left(\lVert (v-\widetilde{v})(0, \cdot) \rVert^2_{H^1(\mathbb{R})}+\lVert (u-\widetilde{u})(0, \cdot) \rVert^2_{H^1(\mathbb{R})}\right).
\end{aligned}
\end{align} 
\end{lemma}
\begin{proof}

Multiplying \eqref{oeh0}$_2$ by $-(u-\widetilde{u})_{xx}$ and integrating the result over $\left[0, t\right]\times \mathbb{R}$, we get
\begin{equation}\label{5result21}
\begin{aligned}
&-\int_{0}^{t}\int_{\mathbb{R}}(u-\widetilde{u})_\tau(u-\widetilde{u})_{xx}dxd\tau \\
=&\int_{0}^{t}\int_{\mathbb{R}} (p(v)-p(\widetilde{v}))_x(u-\widetilde{u})_{xx} dxd\tau-\int_{0}^{t}\int_{\mathbb{R}}\left(\frac{u_x}{v}-\frac{\widetilde{u}_x}{\widetilde{v}}\right)_x(u-\widetilde{u})_{xx}dxd\tau-\int_{0}^{t}\int_{\mathbb{R}}\dot{X}\widetilde{u}_x(u-\widetilde{u})_{xx}dxd\tau \\ 
=&:\mathcal{K}_1+\mathcal{K}_2+\mathcal{K}_3.
\end{aligned}
\end{equation}
Applying the integration by parts to the left-hand side above equation, we have
\begin{equation}
    \begin{aligned} \label{5result2}
        -\int_{0}^{t}\int_{\mathbb{R}}(u-\widetilde{u})_\tau(u-\widetilde{u})_{xx}dxd\tau&=\int_{0}^{t}\int_{\mathbb{R}}\frac{d}{dt}\frac{|(u-\widetilde{u})_x|^2}{2}dxd\tau \\
&=\frac{1}{2}\lVert (u-\widetilde{u})_x(t,\cdot) \rVert^2_{L^2(\mathbb{R})}-\frac{1}{2}\lVert (u-\widetilde{u})_x(0,\cdot) \rVert^2_{L^2(\mathbb{R})}.
    \end{aligned}
\end{equation}

Then, substituting \eqref{5result2} into equation \eqref{5result21} above, we can rewrite that as 
\begin{equation} \label{24}
\frac{1}{2}\lVert (u-\widetilde{u})_x(t,\cdot) \rVert^2_{L^2(\mathbb{R})}=\frac{1}{2}\lVert (u-\widetilde{u})_x(0,\cdot) \rVert^2_{L^2(\mathbb{R})}+\mathcal{K}_1+\mathcal{K}_2+\mathcal{K}_3.
\end{equation}
As in Section 4, we can easily get 
\begin{equation*}
\mathcal{K}_1+\mathcal{K}_2+\mathcal{K}_3\le -\frac{1}{2}\int_{0}^{t} \lVert \frac{(u-\widetilde{u})_{xx}}{v} \rVert_{L^2(\mathbb{R})}^2 d\tau+C\int_{0}^{t}D_vd\tau+C\left(\delta+\varepsilon\right)\int_{0}^{t}\left(\delta|\dot{X}|^2+G_1+G^S+D_{u_1}\right)d\tau.
\end{equation*}
Therefore, combining these and \eqref{lem5.1} as in Lemma 5.1 with the fact that $D_{u_2}\sim \int_{0}^{t} \lVert \frac{(u-\widetilde{u})_{xx}}{v} \rVert_{L^2(\mathbb{R})}^2 d\tau$, we can get the desired estimate \eqref{lem 5.2}.
\end{proof}

\begin{appendix}
    \section{Proof of Theorem \ref{thm:main} }
    \label{Diffusion proof}
		\setcounter{equation}{0}
  \subsection{Global existence of perturbed solution}
	\setcounter{equation}{0}
	Using the a priori estimate \eqref{a-priori-1}, we can extend the local solution obtained from Proposition \ref{prop:local} to the global one by using the standard continuation argument. We first choose smooth functions $\underline{v}$ and $\underline{u}$ that satisfy
	\begin{equation}\label{ubarvbar}
	\sum_{\pm}\left(\|\underline{v}-v_\pm\|_{L^2(\R_\pm)}+\|\underline{u}-u_\pm\|_{L^2(\R_\pm)}\right)+\|\pa_x\underline{v}\|_{L^2(\R)}+\|\pa_x\underline{u}\|_{L^2(\R)}\le C\delta.
	\end{equation}
	
	Then, we use the estimates on the shock wave \eqref{shock-property} to obtain
	\begin{equation}
	\begin{aligned}\label{est-init}
	&\norm{\underline{v}(\cdot)-\tv(0,\cdot)}_{H^1(\mathbb{R})}+\norm{\underline{u}(\cdot)-\tu(0,\cdot)}_{H^1(\mathbb{R})} \\
 &\le \sum_{\pm}\left(\norm{\underline{v}-v_{\pm}}_{L^2(\R_\pm)}+\norm{\underline{u}-u_{\pm}}_{L^2(\R_\pm)}\right)  + \|\tv-v_+\|_{L^2(\R_+)}+\|\tv-v_-\|_{L^2(\R_-)} \\
 &\quad+\|\pa_x \underline{v}\|_{L^2(\mathbb{R})}+ \|\partial_x\tv\|_{L^2(\mathbb{R})} +\|\tu-u_+\|_{L^2(\R_+)}+\|\tu-u_-\|_{L^2(\R_-)}+ \|\pa_x \underline{u}\|_{L^2(\mathbb{R})}+\|\partial_x\tu\|_{L^2(\mathbb{R})}\\
	&\le C\sqrt{\delta} 
	\end{aligned}
	\end{equation}
	Now, for sufficiently small $\delta$ we choose $\e_0$ as 
	\[\e_0 < \frac{\e_1}{3}-C\sqrt{\delta}.\]
	Consider any initial data $(v_0,u_0)$ such that
	\[\sum_{\pm}\left(\|v_0-v_{\pm}\|_{L^2(\R_\pm)}+\|u_0-u_\pm\|_{L^2(\R_\pm)}\right)+\|\partial_x v_{0}\|_{L^2(\mathbb{R})}+\|\partial_x u_{0}\|_{L^2(\mathbb{R})}<\e_0.\]
	Then, we use \eqref{ubarvbar} to obtain
	\begin{align*}
	&\|v_0-\underline{v}\|_{H^1(\mathbb{R})}+\|u_0-\underline{u}\|_{H^1(\mathbb{R})}\\
	&\le\sum_{\pm}\left(\|v_0-v_\pm\|_{L^2(\R_\pm)}+\|u_0-u_{\pm}\|_{L^2(\R_\pm)}+\|\underline{v}-v_\pm\|_{L^2(\R_\pm)}+\|\underline{u}-u_{\pm}\|_{L^2(\R_\pm)}\right)\\
	&\quad + \|\partial_x v_{0}\|_{L^2(\mathbb{R})}+\|\partial_x u_{0}\|_{L^2(\mathbb{R})}+\|\partial_x \underline{v}\|_{L^2(\mathbb{R})}+\|\partial_x \underline{u}\|_{L^2(\mathbb{R})}\\
	&\le \e_0 +C\sqrt{\delta}< \frac{\e_1}{3}.
	\end{align*}
	From the smallness of $\e_1$ and Sobolev embedding, we have
	\[\frac{v_-}{2}\le v_0(x)\le 2v_+,\quad x\in\R\]
	and by the local existence result in Proposition \ref{prop:local}, there exists $T_0>0$ such that
	\begin{equation}\label{est-local-1}
	\|v-\underline{v}\|_{L^\infty(0,T_0;H^1(\mathbb{R}))}+\|u-\underline{u}\|_{L^\infty(0,T_0;H^1(\mathbb{R}))}\le \frac{\e_1}{2}
	\end{equation}
	and
	\[\frac{v_-}{3}\le v(t,x)\le 3v_+,\quad (t,x)\in[0,T_0]\times\R.\]
	On the other hand, we estimate the difference between $(\underline{v},\underline{u})$ and $(\tv^X,\tu^X)$ by using similar estimate as in \eqref{est-init} and \eqref{shock-property} as 
	\begin{align*}
	&\|\underline{v}-\tv^X(t,\cdot)\|_{H^1(\mathbb{R})}+\|\underline{u}-\tu^X(t,\cdot)\|_{H^1(\mathbb{R})}\\
	&\le \sum_{\pm}\left(\|\underline{v}-v_{\pm}\|_{L^2(\R_\pm)}+\|\underline{u}-u_{\pm}\|_{L^2(\R_\pm)}+\|\tv^X-v_\pm\|_{L^2(\R_\pm)}+\|\tu^X-u_\pm\|_{L^2(\R_\pm)}\right)\\
	&\quad +\|\partial_x \tv^X\|_{L^2(\mathbb{R})}+\|\partial_x \tu^X\|_{L^2(\mathbb{R})}+\|\partial_x \underline{v}\|_{L^2(\mathbb{R})}+\|\underline{u}\|_{L^2(\mathbb{R})}\\
	&\le C\sqrt{\delta}(1+\sqrt{|X(t)|})\le C\sqrt{\delta}(1+\sqrt{t}).
	\end{align*}
	Taking $T_1\in(0,T_0)$ small enough so that $C\sqrt{\delta}(1+\sqrt{T_1})\le \frac{\e_1}{2}$, we have
	\begin{equation}\label{est-local-2}
	\|\underline{v}-\tv^X\|_{L^\infty(0,T_1;H^1(\mathbb{R}))}+\|\underline{u}-\tu^X\|_{L^\infty(0,T_1;H^1(\mathbb{R}))}\le\frac{\e_1}{2}.
	\end{equation}
	Combining \eqref{est-local-1} and \eqref{est-local-2} yields
	\[\|v-\tv^X\|_{L^\infty(0,T_1;H^1(\mathbb{R}))}+\|u-\tu^X\|_{L^\infty(0,T_1;H^1(\mathbb{R}))}\le\e_1.\]
	Then, the a priori estimate \eqref{a-priori-1} implies that $T_1$ can be extended to $+\infty$, and the global existence is proved. In particular, we have
	\begin{equation}
	\begin{aligned}\label{est-infinite}
	& \sup_{t>0}\left(\norm{v-\tv^X}_{H^1 (\mathbb{R})}^2 +\norm{u-\tu^X}_{H^1(\mathbb{R})}^2\right)+\delta \int_0^\infty | \dot{X}(t)|^2 \, d t \\ 
	&\quad +\int_0^\infty \left( G_1+G^S \right) \, dt+  \int_0^\infty \left( D_v+D_{u_1} + D_{u_2}  \right)\, dt  \\ 
	& \le C_0 \left(\norm{v_0-\tv}_{H^1 (\mathbb{R})}^2 +\norm{u_0-\tu}_{H^1(\mathbb{R})}^2\right)<\infty
	\end{aligned}
	\end{equation}
	and, for $t>0$,
	\begin{equation} \label{eq: X bound}
	|\dot{X}(t)|\le C_0 \left(\|(v-\tv^X)(t,\cdot)\|_{L^\infty(\bbr)}+\|(u-\tu^X)(t,\cdot)\|_{L^\infty(\bbr)}\right).
	\end{equation}
	
	\subsection{Time-asymptotic behavior}
	We are now ready to prove the time-asymptotic behavior of the perturbation. We first define
	\[g(t):=\|(v-\tv^X)_x\|_{L^2(\mathbb{R})}^2+\|(u-\tu^X)_x\|_{L^2(\mathbb{R})}^2.\]
	We will show that $g\in W^{1,1}(\R_+)$ which implies $\lim_{t\to\infty}g(t)=0$. Then, the Gagliardo-Nirenberg interpolation inequality and the uniform bound estimate \eqref{est-infinite} implies
	\begin{equation} \label{eq: v,u limit}
	\lim_{t\to\infty}\left(\|v-\tv^X\|_{L^\infty(\mathbb{R})}+\|u-\tu^X\|_{L^\infty(\mathbb{R})}\right)=0.
	\end{equation}
	Furthermore, \eqref{eq: X bound} and \eqref{eq: v,u limit} imply that
	\[ \lim_{t \to \infty} |\dot{X}(t)| \le C_0 \lim_{t \to \infty} \left( \norm{(v-\tv^X)(t,\cdot)}_{L^\infty(\mathbb{R})} + \norm{(u-\tu^X)(t,\cdot)}_{L^\infty(\mathbb{R})} \right)=0.\]
	Therefore, it remains to show that $g\in W^{1,1}(\R_+)$.\\
	
	\noindent (1) $g\in L^1(\bbr_+)$: By using \eqref{est-infinite} in the final inequality, we obtain

	\begin{align*}
	\int_0^\infty |g(t)|\,dt&=\int_0^\infty\int_{\R} |(v-\tv^X)_x|^2+|(u-\tu^X)_x|^2\,dxdt\le \int_0^\infty (D_v+D_{u_1})\,d t<\infty,
	\end{align*}
 which implies $g\in L^1(\R_+)$.
	
	\noindent (2) $g'\in L^1(\R_+)$: We combine the system \eqref{eq:NS} and \eqref{viscous-shock} to obtain
	\begin{align}
	\begin{aligned}\label{eq:diff}
	&(v-\tv^X)_t - (u-\tu^X)_x = \dot{X}(t)\tv_x^X,\\
	&(u-\tu^X)_t + (p(v)-p(\tv^X))_x = \left(\frac{u_x}{v}-\frac{\tu_x^X}{\tv^X}\right)_x+\dot{X}(t)\tu_x^X.
	\end{aligned}
	\end{align}
 In addition, we observe that
    \begin{align}
	\begin{aligned} \label{diffusion p}
	|(p(v)-p(\tv^X))_x| &\le |p'(v) | |(v-\tv^X)_x| + |\tv^X_x| |p'(v)-p'(\tv^X)| \le C |(v-\tv^X)_x| + |\tv^X_x| |v-\tv^X| \\ 
 &\le C |(v-\tv^X)_x| + |\tv^X_x| |p(v)-p(\tv^X )|\\ 
 &\le  C |(v-\tv^X)_x| + |\tv^X_x| |p(v)-p(\tv^X ) -\frac{u-\tu^X}{2C_*}| + C|\tv^X_x| |u-\tu^X|
	\end{aligned}
	\end{align}
	Then, using the equation \eqref{eq:diff} and \eqref{diffusion p}, we estimate the time-integration of $g'$ as
	\begin{equation}
	\begin{aligned}\label{est:gprime}
	\int_0^\infty|g'(t)|\,dt&=\int_0^\infty 2\left|\int_\R (v-\tv^X)_x(v-\tv^X)_{xt}\,dx+\int_\R (u-\tu^X)_x(u-\tu^X)_{xt}\,dx\right|\,dt\\
	&\le 2\int_0^\infty \left|\int_\R (v-\tv^X)_x\left((u-\tu^X)_{xx}+\dot{X}(t)\tv^X_{xx}\right)\,dx\right|\,dt\\
	&\quad + 2\int_0^\infty \Bigg|\int_\R (u-\tu^X)_x \left(-(p(v)-p(\tv^X))_{xx}+\left(\frac{u_x}{v}-\frac{\tu_x^X}{\tv^X}\right)_{xx}+\dot{X}(t)\tu_{xx}^X\right)\Bigg|\,dt\\
	&=2\int_0^\infty \left|\int_\R (v-\tv^X)_x\left((u-\tu^X)_{xx}+\dot{X}(t)\tv^X_{xx}\right)\,dx\right|\,dt\\
	&\quad + 2\int_0^\infty \Bigg|\int_\R (u-\tu^X)_{xx}\left((p(v)-p(\tv^X))_{x}-\left(\frac{u_x}{v}-\frac{\tu_x^X}{\tv^X}\right)_{x}\right) + \int_\R (u-\tu^X)_x\dot{X}(t)\tu^X_{xx}\Bigg|\,dt\\
	&\le C\int_0^\infty (D_{v}+D_{u_1}+D_{u_2}+|\dot{X}(t)|^2)\,dt + C\int_0^\infty \int_\R \left|\left(\frac{u_x}{v}-\frac{\tu_x^X}{\tv^X}\right)_{x}\right|^2 \,dxdt.
	\end{aligned}
	\end{equation}
	
	Since the first term in the right-hand side of \eqref{est:gprime} can be bounded by \eqref{est-infinite}, we only need to estimate the last term.   Precisely, we obtain
	\begin{align*}
	\int_0^\infty& \int_\R \left|\left(\frac{u_x}{v}-\frac{\tu^X_x}{\tv^X}\right)_x\right|^2\,dxdt\\
	&=\int_0^\infty\int_\R \Bigg|\frac{1}{v}(u-\tu^X)_{xx}+\tu^X_{xx}\left(\frac{1}{v}-\frac{1}{\tv^X}\right)-\frac{1}{v^2}(v-\tv^X)_{x}(u-\tu^X)_x\\
	&\hspace{2cm}-\frac{\tv^X_x}{v^2}(u-\tu^X)_x-\frac{\tu^X_x}{v^2}(v-\tv^X)_x-\tv^X_x\tu^X_x\left(\frac{1}{v^2}-\frac{1}{(\tv^X)^2}\right)\Bigg|^2\,dxdt\\
	&\le C\int_0^\infty\int_\R \Bigg(|(u-\tu^X)_{xx}|^2 + |\tu_x^X|^2|v-\tv^X|^2+|(u-\tu^X)_x|^2|(v-\tv^X)_x|^2\\
	&\hspace{3cm}+|\tv_x^X|^2|(u-\tu^X)_x|^2+|\tu_x^X|^2|(v-\tv^X)_x|^2+|\tv^X_x|^2|\tu^X_x|^2|v-\tv^X|^2\Bigg)\,dxdt,
	\end{align*}
	and consequently,
	\begin{align*}
	\int_0^\infty& \int_\R \left|\left(\frac{u_x}{v}-\frac{\tu^X_x}{\tv^X}\right)_x\right|^2\,dxdt\\
	&\le C\int_0^\infty\left(G_1+G^S+D_v+D_{u_1}+D_{u_2}\right)\,dt\\
	&\quad + C\|(v-\tv^X)_x\|^2_{L^\infty((0,\infty)\times\R)}\int_0^\infty\int_\R |(u-\tu^X)_x|^2\,dxdt\\
	&\le C\int_0^\infty\left(G_1+G^S+D_v+D_{u_1}+D_{u_2}\right)\,dt+ C\int_0^\infty D_{u_1}\,dt<+\infty.
	\end{align*}

	This proves $g'\in L^1(\R_+)$. Thus, we have shown that $g\in W^{1,1}(\R_+)$. This completes the proof of the asymptotic behavior of the NS equations. Therefore, once we have the a priori estimate in Proposition \ref{apriori-estimate}, we prove the time asymptotic behavior of the NS equations. \\
\end{appendix}
\bibliographystyle{amsplain}
\bibliography{reference} 

\providecommand{\bysame}{\leavevmode\hbox to3em{\hrulefill}\thinspace}
\providecommand{\MR}{\relax\ifhmode\unskip\space\fi MR }
\providecommand{\MRhref}[2]{%
  \href{http://www.ams.org/mathscinet-getitem?mr=#1}{#2}
}
\providecommand{\href}[2]{#2}
\begin{thebibliography}{10}

\bibitem{CKKV}
K.~Choi, M.-J. Kang, Y.-S. Kwon, and A.~F. Vasseur, \emph{Contraction for large
  perturbations of traveling waves in a hyperbolic-parabolic system arising
  from a chemotaxis model}, Math. Models Methods Appl. Sci. \textbf{30} (2020),
  no.~2, 387--437.

\bibitem{CKV}
K.~Choi, M.-J. Kang, and A.~F. Vasseur, \emph{Global well-posedness of large
  perturbations of traveling waves in a hyperbolic-parabolic system arising
  from a chemotaxis model}, J. Math. Pures Appl. (9) \textbf{142} (2020),
  266--297.

\bibitem{G86}
J.~Goodman, \emph{Nonlinear asymptotic stability of viscous shock profiles for
  conservation laws}, Arch. Rational Mech. Anal. \textbf{95} (1986), no.~4,
  325--344.

\bibitem{HKK23}
S.~Han, M.-J. Kang, and J.~Kim, \emph{Large-time behavior of composite waves of
  viscous shocks for the barotropic {N}avier-{S}tokes equations}, SIAM J. Math.
  Anal. \textbf{55} (2023), no.~5, 5526--5574.

\bibitem{HKKL}
S.~Han, M.-J. Kang, J.~Kim, and H.~Lee, \emph{{Long-time behavior towards
  viscous-dispersive shock for Navier-Stokes equations of Korteweg type}},
  arXiv preprint arXiv:2402.09751 (2024).

\bibitem{XHKKL}
X.~Huang, J.~Kim, and H.~Lee, \emph{{Asymptotic behavior toward viscous shock
  for impermeable wall and inflow problems of barotropic Navier-Stokes
  equations}}, arXiv preprint arXiv:2405.03214 (2024).

\bibitem{Kang19}
M.-J. Kang, \emph{{$L^2$}-type contraction for shocks of scalar viscous
  conservation laws with strictly convex flux}, J. Math. Pures Appl. (9)
  \textbf{145} (2021), 1--43.

\bibitem{KO}
M.-J. Kang and H.~Oh, \emph{{$L^2$ decay for large perturbations of viscous
  shocks for multi-D Burgers equation}}, arXiv preprint arXiv:2403.08445
  (2024).

\bibitem{KV16}
M.-J. Kang and A.~F. Vasseur, \emph{Criteria on contractions for entropic
  discontinuities of systems of conservation laws}, Arch. Ration. Mech. Anal.
  \textbf{222} (2016), no.~1, 343--391.

\bibitem{Kang-V-1}
\bysame, \emph{{$L^2$}-contraction for shock waves of scalar viscous
  conservation laws}, Ann. Inst. H. Poincar\'{e} C Anal. Non Lin\'{e}aire
  \textbf{34} (2017), no.~1, 139--156.

\bibitem{KV21}
\bysame, \emph{Contraction property for large perturbations of shocks of the
  barotropic {N}avier-{S}tokes system}, J. Eur. Math. Soc. \textbf{23} (2021),
  no.~2, 585--638.

\bibitem{KV-Inven}
\bysame, \emph{Uniqueness and stability of entropy shocks to the isentropic
  {E}uler system in a class of inviscid limits from a large family of
  {N}avier-{S}tokes systems}, Invent. Math. \textbf{224} (2021), no.~1,
  55--146.

\bibitem{KV-2shock}
\bysame, \emph{Well-posedness of the {R}iemann problem with two shocks for the
  isentropic {E}uler system in a class of vanishing physical viscosity limits},
  J. Differential Equations \textbf{338} (2022), 128--226.

\bibitem{KVW}
M.-J. Kang, A.~F. Vasseur, and Y.~Wang, \emph{{$L^2$}-contraction of large
  planar shock waves for multi-dimensional scalar viscous conservation laws},
  J. Differential Equations \textbf{267} (2019), no.~5, 2737--2791.

\bibitem{KVW23}
\bysame, \emph{Time-asymptotic stability of composite waves of viscous shock
  and rarefaction for barotropic {N}avier-{S}tokes equations}, Adv. Math.
  \textbf{419} (2023), Paper No. 108963, 66.

\bibitem{KVW-NSF}
\bysame, \emph{Time-asymptotic stability of generic riemann solutions for
  compressible navier-stokes-fourier equations}, arXiv preprint
  arXiv:2306.05604 (2023).

\bibitem{M18}
A.~Matsumura, \emph{Waves in compressible fluids: viscous shock, rarefaction,
  and contact waves}, Handbook of mathematical analysis in mechanics of viscous
  fluids, Springer, Cham, 2018, pp.~2495--2548.

\bibitem{MN85}
A.~Matsumura and K.~Nishihara, \emph{On the stability of travelling wave
  solutions of a one-dimensional model system for compressible viscous gas},
  Japan J. Appl. Math. \textbf{2} (1985), no.~1, 17--25.

\bibitem{MN86}
\bysame, \emph{Asymptotics toward the rarefaction waves of the solutions of a
  one-dimensional model system for compressible viscous gas}, Japan J. Appl.
  Math. \textbf{3} (1986), no.~1, 1--13.

\bibitem{MW10}
A.~Matsumura and Y.~Wang, \emph{Asymptotic stability of viscous shock wave for
  a one-dimensional isentropic model of viscous gas with density dependent
  viscosity}, Methods Appl. Anal. \textbf{17} (2010), no.~3, 279--290.

\bibitem{MV}
A.~Mellet and A.~Vasseur, \emph{On the barotropic compressible navier–stokes
  equations}, Communications in Partial Differential Equations \textbf{32}
  (2007), no.~3, 431--452.

\bibitem{S76}
V.~A. Solonnikov, \emph{The solvability of the initial-boundary value problem
  for the equations of motion of a viscous compressible fluid}, Zap. Nau\v{c}n.
  Sem. Leningrad. Otdel. Mat. Inst. Steklov. (LOMI) \textbf{56} (1976),
  128--142, 197, Investigations on linear operators and theory of functions,
  VI.

\bibitem{Vasseur-2013}
A.~Vasseur, \emph{Relative entropy and contraction for extremal shocks of
  conservation laws up to a shift}, Recent advances in partial differential
  equations and applications, Contemp. Math., vol. 666, Amer. Math. Soc.,
  Providence, RI, 2016, pp.~385--404. \MR{3537479}

\bibitem{WW}
T.~Wang and Y.~Wang, \emph{Nonlinear stability of planar viscous shock wave to
  three-dimensional compressible navier-stokes equations}, arXiv preprint
  arXiv:2204.09428 (2022).

\end{thebibliography}
\end{document}